\documentclass[11pt,reqno]{amsart}

\usepackage{amssymb, amsmath, amsthm}
\usepackage{hyperref}
\usepackage[alphabetic,lite]{amsrefs}
\usepackage{verbatim}
\usepackage{amscd}   
\usepackage[all]{xy} 
\usepackage{youngtab} 
\usepackage{young} 
\usepackage{tikz}
\usepackage{ mathrsfs }
\usepackage{cases}

\newcommand{\defi}[1]{{\upshape\sffamily #1}}

\renewcommand{\d}{\delta}
\newcommand{\D}{\mathcal{D}}
\newcommand{\bw}{\bigwedge}
\renewcommand{\det}{\textrm{det}}
\newcommand{\G}{\Gamma}

\renewcommand{\ll}{\lambda}
\newcommand{\m}{\mathfrak{m}}

\newcommand{\om}{\omega}

\newcommand{\oo}{\otimes}

\newcommand{\gr}{\textrm{gr}}
\newcommand{\rank}{\textrm{rank}}

\newcommand{\Ext}{\operatorname{Ext}}
\newcommand{\GL}{\operatorname{GL}}

\newcommand{\SL}{\operatorname{SL}}

\newcommand{\Spec}{\operatorname{Spec}}
\newcommand{\Sym}{\operatorname{Sym}}

\newcommand{\Ver}{\operatorname{Ver}}

\newcommand{\bb}[1]{\mathbb{#1}}

\renewcommand{\rm}[1]{\textrm{#1}}
\newcommand{\mc}[1]{\mathcal{#1}}

\newcommand{\ol}[1]{\overline{#1}}

\newcommand{\tl}[1]{\tilde{#1}}
\newcommand{\ul}[1]{\underline{#1}}
\newcommand{\scpr}[2]{\left\langle #1,#2 \right\rangle}

\def\lra{\longrightarrow}

\newtheorem{theorem}{Theorem}[section]
\newtheorem*{theorem*}{Theorem}
\newtheorem{lemma}[theorem]{Lemma}

\newtheorem{proposition}[theorem]{Proposition}
\newtheorem{corollary}[theorem]{Corollary}
\newtheorem*{corollary*}{Corollary}

\newtheorem*{charext*}{Theorem on the Characters of Ext Modules}
\newtheorem*{ext*}{Theorem on the Growth of Ext Modules}
\newtheorem*{loccoh*}{Theorem on Local Cohomology with Support in Generic Determinantal Ideals}
\newtheorem*{regularity*}{Theorem on Regularity of Equivariant Ideals}
\newtheorem*{nonvanishing*}{Theorem on Non-vanishing of Local Cohomology with Determinantal Support}

\theoremstyle{definition}

\newtheorem*{definition*}{Definition}

\theoremstyle{remark}

\newtheorem*{remark*}{Remark}

\numberwithin{equation}{section}

\begin{document}

\title{Characters of equivariant $\D$-modules on Veronese cones}

\author{Claudiu Raicu}
\address{Department of Mathematics, University of Notre Dame, 255 Hurley, Notre Dame, IN 46556\newline
\indent Institute of Mathematics ``Simion Stoilow'' of the Romanian Academy}
\email{craicu@nd.edu}

\subjclass[2010]{Primary 13D45, 14M17, 14F10, 14F40}

\date{\today}

\keywords{$\D$-modules, Veronese cones, local cohomology}

\begin{abstract} For $d>1$, we consider the Veronese map of degree $d$ on a complex vector space $W$, $\Ver_d:W\lra\Sym^d W$, $w\mapsto w^d$, and denote its image by $Z$. We describe the characters of the simple $\GL(W)$-equivariant holonomic $\D$-modules supported on~$Z$. In the case when $d=2$, we obtain a counterexample to a conjecture of Levasseur by exhibiting a $\GL(W)$-equivariant $\D$-module on the Capelli type representation $\Sym^2 W$ which contains no $\SL(W)$-invariant sections. We also study the local cohomology modules $H^{\bullet}_Z(S)$, where $S$ is the ring of polynomial functions on the vector space $\Sym^d W$. We recover a result of Ogus showing that there is only one local cohomology module that is non-zero (namely in degree $\bullet=\rm{codim}(Z)$), and moreover we prove that it is a simple $\D$-module and determine its character.
\end{abstract}

\maketitle

\section{Introduction}\label{sec:intro}

Given a complex vector space $W$ of dimension $n=\dim(W)$, and an integer $d>1$, we let $\Ver_d:W\lra\Sym^d W$, $w\to w^d$, denote the \defi{degree $d$ Veronese map}. Its image $Z$, called the \defi{degree $d$ Veronese cone}, can be identified with the set of $d$-th powers of linear forms on the dual vector space $V=W^*$. The group $\GL(W)\simeq\GL_n(\bb{C})$ of invertible linear transformations of $W$ acts on $\Sym^d W$ preserving the Veronese cone~$Z$. We write $\mc{D}$ for the \defi{Weyl algebra} of differential operators with polynomial coefficients on the vector space $\Sym^d W$. In this paper we prove two main results:
\begin{enumerate}
 \item We describe the structure as $\GL(W)$-representations (\defi{the characters}) of the simple $\GL(W)$-equivariant holonomic $\mc{D}$-modules whose support is $Z$. In the special case when $d=2$, our calculation provides a counterexample to a conjecture of Levasseur \cite[Conjecture~5.17]{levasseur}.
 \item Letting $S=\Sym(\Sym^d V)$ denote the ring of polynomial functions on $\Sym^d W$, we show that the unique non-vanishing local cohomology module of $S$ with support in $Z$ (namely $H_Z^{\rm{codim}(Z)}(S)$) is one of the $\D$-modules in~(1), and in particular we obtain a description of its character.
\end{enumerate}

It follows from the equivariant version of the Riemann--Hilbert correspondence (see \cite[Section~11.6]{hottaetal} and Section~\ref{sec:equivDmodsVero}) that there are precisely $(d+1)$ simple $\GL(W)$-equivariant $\D$-modules whose support is contained in $Z$, which we denote by
\[E,D_0,D_1,\cdots,D_{d-1}.\]
For the experts, we note that $D_0$ is the $\D$-module corresponding to the intersection cohomology sheaf arising from the trivial local system on the orbit $Z\setminus\{0\}$ (whose fundamental group is cyclic of order $d$), while $D_1,\cdots,D_{d-1}$ correspond to the simple non-trivial local systems. Each $D_j$ has support equal to $Z$, while $E$ is supported at the origin and is very well understood. Our main focus will be to compute the characters of the modules $D_j$, but before that we recall several descriptions of $E$:
\begin{itemize}
 \item If we identify $\Sym^d W$ with $\bb{C}^N$, where $N=\dim(\Sym^d W)$, and write $\mc{D}=\bb{C}[x_1,\cdots,x_N,\partial_1,\cdots,\partial_N]$, with $\partial_i=\partial/\partial x_i$, then $E\simeq\D/(x_1,\cdots,x_N)$.
 \item $E$ is the injective envelope of the residue field $\bb{C}=S/\m$, where $\m=(x_1,\cdots,x_N)$ is the maximal homogeneous ideal of $S=\bb{C}[x_1,\cdots,x_N]$.
 \item $E$ is the local cohomology module $H_{\m}^N(S)$.
 \item $E$ is the graded dual of the polynomial ring $S$.
 \item $E$ is the Fourier transform of the $\D$-module $S$.
 \item The structure of $E$ as a $\GL(W)$-representation is given by
\[E=\det(\Sym^d W)\oo\Sym(\Sym^d W),\]
where $\det(\Sym^d W)=\bigwedge^N(\Sym^d W)$ denotes the top exterior power of $\Sym^d W$.
\end{itemize}

We begin by stating our main result in the case when $d=2$, where it is most explicit. In this case, it follows from \cite[Exercise~I.8.6(a)]{macdonald} that the module $E=\bigoplus_{\ll}S_{\ll}W$ has a multiplicity free decomposition as a $\GL(W)$-representation, where $\ll$ runs over the set of dominant weights for which $\ll_i\geq n+1$ and $\ll_i-n$ is odd for every $i=1,\cdots,n$. For the equivariant $\D$-modules with support $Z$ we get:

\begin{theorem}\label{thm:charDmods2}
 There are two simple $\GL(W)$-equivariant $\mc{D}$-modules $D_0,D_1$ whose support is the degree $d=2$ Veronese cone $Z$. They have a multiplicity free decomposition into irreducible $\GL(W)$-representations
\[D_j=\bigoplus_{\ll}S_{\ll}W,\]
where $\ll=(\ll_1\geq\cdots\geq\ll_n)$ runs over a set of dominant weights in $\bb{Z}^n$ as follows:
\[
\begin{array}{c|c|c}
 & n\rm{ even} & n\rm{ odd}\\ 
\hline
 & & \\
 & \ll_1,\cdots,\ll_{n-1}\geq n & \ll_1,\cdots,\ll_{n-1}\geq n,\ \ll_n\leq n-1 \\
D_0 & & \\
 & all\ \ll_i\ even & \ll_i\rm{ odd for }i\leq n-1,\ \ll_n\rm{ even} \\
 & & \\
\hline
 & & \\
 & \ll_1,\cdots,\ll_{n-1}\geq n,\ \ll_n\leq n-1 & \ll_1,\cdots,\ll_{n-1}\geq n, \\
D_1 & & \\
 & \ll_i\rm{ even for }i\leq n-1,\ \ll_n\rm{ odd} & all\ \ll_i\ odd \\
 & & \\
\end{array}
\]
Furthermore, the local cohomology modules $H_Z^{\bullet}(S)$ of $S=\Sym(\Sym^2 V)$ with support in $Z$ are given by:
\[
 H_Z^{\bullet}(S)=
 \begin{cases}
  D_0, & \bullet = {n\choose 2}. \\
  0, & otherwise.
 \end{cases}
\]
\end{theorem}

The elements of the $\D$-module $D_j$ which are invariant with respect to the action of the special linear group $\SL(W)$ (which we denote by $D_j^{\SL(W)}$) correspond to dominant weights $\ll$ with $\ll_1=\cdots=\ll_n$. It follows from the description in Theorem~\ref{thm:charDmods2} that when $n$ is even $D_1^{\SL(W)}=0$, while for $n$ odd $D_0^{\SL(W)}=0$. Therefore the Capelli type representations $(\GL(W):\Sym^2 W)$ don't satisfy the conclusion of \cite[Conjecture~5.17]{levasseur}.

If we identify $\Sym^2 W$ with the space of $n\times n$ symmetric matrices, then the degree two Veronese cone is precisely the set of matrices of rank at most one. In \cite{raicu-weyman}, we computed together with Weyman the $\GL$-equivariant structure of the local cohomology modules with support in non-symmetric matrices of arbitrary rank, but our methods there don't directly generalize to symmetric matrices. Nevertheless, the $\D$-module approach explained here for Veronese cones will allow us to overcome the difficulties that arise in the symmetric case: this will be addressed in future work with Weyman. Here we proceed to generalize Theorem~\ref{thm:charDmods2} in a different direction, namely by considering $d>2$.

Given a partition $\mu=(\mu_1\geq\cdots\geq\mu_{n-1}\geq 0)$, we define $\nu_{\mu}=\nu_{\mu}^d$ to be the multiplicity of the irreducible representation $S_{\mu}W$ inside
\[\bigoplus_{2\pi_2+\cdots+d\pi_d=|\mu|} \Sym^{\pi_2}(\Sym^2 W)\oo\cdots\oo\Sym^{\pi_d}(\Sym^d W).\]
The quantity $\nu_{\mu}$ was shown in \cite{manivel} to compute a certain stable multiplicity for symmetric plethysm (see Section~\ref{subsec:stablemults}), and will appear in our description of the characters of the $\D$-modules $D_j$. There are some cases when $\nu_{\mu}$ can be described more explicitly, for instance:
\begin{itemize}
 \item When $d=2$, $\nu_{\mu}=1$ precisely when all the parts $\mu_1,\cdots,\mu_{n-1}$ are even, and $\nu_{\mu}=0$ otherwise.
 \item When $n=2$, $\mu$ is just a number (a partition with one part), and the sequence $(\nu_{\mu})_{\mu\geq 0}$ is encoded by the generating function
\[\sum_{\mu\geq 0}\nu_{\mu}\cdot t^{\mu}=\frac{1}{(1-t^2)\cdots(1-t^d)}.\]
\end{itemize}

Given a dominant weight $\ll=(\ll_1\geq\cdots\geq\ll_n)\in\bb{Z}^n$, we let $u_d=\displaystyle{n-1+d\choose n}$ and define for $1\leq i\leq n$ 
\begin{equation}\label{eq:defll^i}
\ll^i=(\ll_1+1-u_d,\cdots,\ll_{i-1}+1-u_d,\ll_{i+1}-u_d,\cdots,\ll_n-u_d)\in\bb{Z}^{n-1}.
\end{equation}
Note that $u_d\cdot n = N\cdot d$ (where $N=\dim(\Sym^d W)$), so that $\det(\Sym^d W)=S_{(u_d^n)}W$, where $(u_d^n)$ denotes the partition with $n$ parts equal to $u_d$. We make the convention that $\nu_{\mu}=0$ when $\mu\in\bb{Z}^{n-1}$ is not a partition (some $\mu_i<0$) and define
\begin{equation}\label{eq:defmll}
 m_{\ll}=\sum_{i=1}^n (-1)^{n-i}\cdot\nu_{\ll^i}.
\end{equation}
With the notation above, our main result is the following generalization of Theorem~\ref{thm:charDmods2} to arbitrary $d$:

\begin{theorem}\label{thm:main}
 There are $d$ simple $\GL(W)$-equivariant $\mc{D}$-modules $D_0,D_1,\cdots,D_{d-1}$ whose support is the degree $d$ Veronese cone $Z$. Their decomposition into irreducible $\GL(W)$-representations is
\[D_j=\bigoplus_{\ll}(S_{\ll}W)^{\oplus a_{\ll}^j},\]
where $\ll=(\ll_1\geq\cdots\geq\ll_n)\in\bb{Z}^n$ and the multiplicities $a_{\ll}^j$ are zero unless $\ll_1+\cdots+\ll_n\equiv j\rm{ (mod }d)$, in which case they are computed as follows ($e_{\ll}$ denotes the multiplicity of $S_{\ll}W$ inside $E$ and $m_{\ll}$ is as in (\ref{eq:defmll})):
\[
a_{\ll}^j=
\begin{cases}
 m_{\ll}+(-1)^n\cdot e_{\ll}, & j=0; \\
 m_{\ll}, & j=1,\cdots,d-1. \\
\end{cases}
\]
Furthermore, if we let $n_d={n-1+d\choose d}-n$ denote the codimension of $Z$ inside $\Sym^d W$, then we have that
\[
 H_Z^{\bullet}(S)=
 \begin{cases}
  D_0, & \bullet = n_d. \\
  0, & otherwise.
 \end{cases}
\]
\end{theorem}

The vanishing $H_Z^{\bullet}(S)=0$ for $\bullet\neq\rm{codim}(Z)$ when $Z$ is a Veronese cone was first observed in \cite[Example~4.6]{ogus}. The local cohomology modules of $S$ with arbitrary support are holonomic $\D$-modules, and in particular they have finite length, i.e. they admit a composition series with finitely many simple factors. Computing the $\D$-module composition factors for local cohomology modules is typically a difficult problem. For local cohomology modules that are supported at the origin, this can be done in terms of singular cohomology \cite[Theorem~3.1]{lyubeznik-singh-walther}, but we are not aware of good methods of testing for instance whether $E$ appears as a composition factor in $H_Z^{\rm{codim}(Z)}(S)$ when $Z$ has dimension greater than zero. In the case when $Z$ is a Veronese cone, it follows from general principles (see Section~\ref{subsec:loccohDmods}) that in order to prove the equality $H_Z^{\rm{codim}(Z)}(S)=D_0$, it is sufficient to prove that $E$ doesn't occur as a composition factor of $H_Z^{\rm{codim}(Z)}(S)$: we give two independent proofs of this fact, one based on representation theory, and another based on the vanishing of the top de Rham cohomology group of the $\D$-module $H_Z^{\rm{codim}(Z)}(S)$ (we are grateful to Robin Hartshorne for explaining this second approach to us).

We can rewrite the first part of Theorem~\ref{thm:main} more compactly (see Section~\ref{subsec:SLFreps} for the formalism) as
\begin{equation}\label{eq:sumDj}
D_0+\cdots+D_{d-1}=(-1)^n\cdot E+\sum_{\ll=(\ll_1\geq\cdots\geq\ll_n)}m_{\ll}\cdot S_{\ll}W, 
\end{equation}
where $D_j$ collects all the weights $\ll$ on the right hand side that satisfy $\ll_1+\cdots+\ll_n\equiv j\rm{ (mod }d)$. Let us check that indeed, Theorem~\ref{thm:charDmods2} is the special case of Theorem~\ref{thm:main} when $d=2$. We need to understand for which weights $\ll$ is the multiplicity $m_{\ll}\neq 0$: a necessary condition is that some $\nu_{\ll^i}\neq 0$. Note that $u_2=n+1$ and that for each $i=1,\cdots,n$, either $\nu_{\ll^1}=\cdots=\nu_{\ll^{i-1}}=0$, or $\nu_{\ll^{i+1}}=\cdots=\nu_{\ll^n}=0$: this is because $\ll_i-1-n$ appears as a part in $\ll^j$ for $j<i$, and $\ll_i-n$ appears as a part in $\ll^j$ for $j>i$; since $\nu_{\mu}=0$ unless all the parts of $\mu$ are even, the conclusion follows. We conclude that no more than two values $\nu_{\ll^i}$ are different from zero, and that we can get two non-zero values only in consecutive spots, $\nu_{\ll^i}$ and $\nu_{\ll^{i+1}}$. An easy parity argument implies that if $\nu_{\ll^i}=1$ for $1<i<n$, then either $\nu_{\ll^{i-1}}=1$ or $\nu_{\ll^{i+1}}=1$, and in both cases $m_{\ll}=0$. It follows that $m_{\ll}\neq 0$ precisely when the only non-vanishing $\nu_{\ll^i}$ is either $\nu_{\ll^1}$ or $\nu_{\ll^n}$. If $\nu_{\ll^1}$ is the only non-zero term in (\ref{eq:defmll}) then $\ll_n\geq n+1$ and $\ll_i-1-n$ is even for all $i$, and these conditions are precisely equivalent to $e_{\ll}=1$; we get in this case that $m_{\ll}=(-1)^{n-1}$ and $e_{\ll}=1$, so $a_{\ll}^0=m_{\ll}+(-1)^n\cdot e_{\ll}=0$. It follows that the only $\ll$'s for which the right hand side of (\ref{eq:sumDj}) has a non-trivial contribution are the ones for which $\nu_{\ll^n}$ is the only non-zero term appearing in (\ref{eq:defmll}), which is equivalent to saying that $\ll_{n-1}\geq n$, $\ll_i-n$ is even for $i\leq n-1$, and $\ll_n-n$ is even or negative (or both). We get
\[D_0+D_1=\sum_{\substack{\ll_{n-1}\geq n \\ \ll_i-n\rm{ even for }i\leq n-1 \\ \ll_n-n\rm{ even and/or negative}}}S_{\ll}W,\]
with the terms for which $\ll_1+\cdots+\ll_n$ is even contributing to $D_0$, and the rest contributing to $D_1$. This information is more leisurely recorded in the table from Theorem~\ref{thm:charDmods2}.

Our strategy for proving Theorem~\ref{thm:main} is as follows. We consider the resolution of singularities of the Veronese cone $Z$ via the total space $Y$ of the line bundle $\mc{O}(-d)$ on the projective space $\bb{P}V$ (of $1$-dimensional quotients of $V=W^*$). On $Y$ we can write down $\D_Y$-modules $\mc{M}_0,\mc{M}_1,\cdots,\mc{M}_{d-1}$ which are the relative versions of the holonomic $\bb{C}[x,\partial/\partial x]$-modules $x^{-j/d}\cdot\bb{C}[x,1/x]$. We then compute the (Euler characteristics of the) pushforwards of the $\mc{M}_j$'s and deduce from that the characters of the $D_j$'s: the only difficulty arises when pushing forward $\mc{M}_0$, since its pushforward involves (copies of shifts of) both $D_0$ and $E$; to count them, we then use the explicit computation of the Decomposition Theorem for the resolution $Y\to Z$, which is done for instance in \cite[Theorem~6.1]{deCataldo-Migliorini-Mustata}. As far as local cohomology is concerned, we use its description as a limit of $\Ext$ modules, and set up a spectral sequence to compute it. The terms in the spectral sequence are $\GL$-representations and we use them to conclude that $E$ can't appear as a composition factor in the local cohomology modules. This is enough to conclude the proof of the theorem, and it also provides an alternative path to computing the character of $D_0$.

The paper is organized as follows. In Section~\ref{sec:prelim} we collect some preliminary results and fix some notation concerning the representation theory of general linear groups, $\D$-modules and local cohomology. In Section~\ref{sec:equivDmodsVero} we compute the characters of the equivariant $\D$-modules on Veronese cones, while in Section~\ref{sec:loccohVero} we perform the local cohomology calculation.

\section{Preliminaries}\label{sec:prelim}

\subsection{Representation Theory {\cite{ful-har}, \cite[Ch.~2]{weyman}}}\label{subsec:repthy}

Throughout this paper, $W$ will denote a vector space of dimension $\dim(W)=n$ over the field $\bb{C}$ of complex numbers. $\GL(W)$ is the group of invertible linear transformations of $W$, and its irreducible representations are classified by \defi{dominant weights} $\ll=(\ll_1\geq\cdots\geq\ll_n)\in\bb{Z}^n$. We write $S_{\ll}W$ for the irreducible corresponding to $\ll$, and let $(u^n)$ denote the weight with all \defi{parts} $\ll_1,\cdots,\ll_n$ equal to $u$.  A dominant weight $\ll$ is called a \defi{partition} if all its parts are nonnegative. The \defi{determinant} of a $\GL(W)$-representation $U$ is its top exterior power, $\det(U)=\bw^{\dim U}U$. We have $\det(W)=S_{(1^n)}W$, $S_{\ll}W\oo\det(W)=S_{\ll+(1^n)}W$ and $S_{\ll}W=S_{(-\ll_n,\cdots,-\ll_1)}V$, where $V=W^*$ is the dual vector space. For any $\GL(W)$-representation $U$, $\det(U)$ is isomorphic to some power of $\det(W)$, i.e. $\det(U)\simeq S_{(u^n)}W$ for an appropriate $u$. When $U=\Sym^d W$,
\begin{equation}\label{eq:defud}
 \det(\Sym^d W)=\det(W)^{u_d}=S_{(u_d^n)}W,\quad\rm{where}\quad u_d={n-1+d\choose n}.
\end{equation}
If $U=\bigoplus_{\ll}(S_{\ll}W)^{\oplus a_{\ll}}$ is a representation of $\GL(W)$, we write
\begin{equation}\label{eq:defscpr}
\scpr{S_{\ll}W}{U}=a_{\ll} 
\end{equation}
for the multiplicity of the irreducible $S_{\ll}W$ inside $U$. We write $U^{\ll}$ for the subrepresentation $(S_{\ll}W)^{\oplus a_{\ll}}$, and call it the \defi{$\ll$-isotypic component} of $U$.

The \defi{size} of $\ll$ is defined by $|\ll|=\ll_1+\cdots+\ll_n$. If $\mu\in\bb{Z}^{n-1}$ and $r\in\bb{Z}$, we write
\[\mu[r]=(r-|\mu|,\mu_1,\cdots,\mu_{n-1})\in\bb{Z}^n.\]
If $\mu$ is a dominant weight and $r$ is sufficiently large, then $\mu[r]$ is also dominant.

\subsection{Stable multiplicities for symmetric plethysm {\cite{manivel}}}\label{subsec:stablemults}

The multiplicities of the Schur functors appearing in the decomposition of the plethysm $\Sym^r(\Sym^d \bb{C}^{n})$ are notoriously hard to compute. Nevertheless, Manivel proved a stabilization result that will turn out to be useful for our investigations:

\begin{theorem}[\cite{manivel}]
 Fix a partition $\mu=(\mu_1,\cdots,\mu_{n-1})$ and define, for $k\geq 0$,
\[\nu_{\mu}(k)=\scpr{S_{\mu[kd]}\bb{C}^{n}}{\Sym^k(\Sym^d\bb{C}^{n})}.\]
The sequence $(\nu_{\mu}(k))_{k\geq 0}$ is non-decreasing and stabilizes to $\nu_{\mu}=\nu_{\mu}(k)$ for $k\gg 0$, where
\begin{equation}\label{eq:defnumu}
\begin{aligned}
\nu_{\mu} &= \scpr{S_{\mu}\bb{C}^{n-1}}{\bigoplus_{2\pi_2+\cdots+d\pi_d=|\mu|} \Sym^{\pi_2}(\Sym^2\bb{C}^{n-1})\oo\cdots\oo\Sym^{\pi_d}(\Sym^d\bb{C}^{n-1})} \\
 &= \scpr{S_{\mu}\bb{C}^{n-1}}{\bigotimes_{k=2}^d\Sym(\Sym^k\bb{C}^{n-1})}.
\end{aligned}
\end{equation}
\end{theorem}

In fact, the result of Manivel is more precise, and we reformulate it in a way that's more suitable for our purposes. Let $S=\Sym(\Sym^d V)$, and write
\[S=\bigoplus_{\ll}{S^{\ll}},\]
where $S^{\ll}$ denotes as before the $\ll$-isotypic component of $S$. We have $S^{\ll}=(S_{\ll}V)^{\oplus s_{\ll}}$, where $s_{\ll}=\scpr{S_{\ll}V}{S}$ (using the notation (\ref{eq:defscpr})). It is easy to see that $s_{\ll}=0$ unless $\ll$ is a partition ($\ll_n\geq 0$) and $|\ll|\equiv 0\ (\rm{mod }d)$. One can show (see \cite[Theorem~3.1]{langley-remmel} for a slightly more general statement) that when $\ll$ is a \defi{hook partition} (i.e. when $\ll_2\leq 1$), we have
\begin{equation}\label{eq:sllhook}
 s_{\ll}=\begin{cases}
1, & \rm{if}\ \ll=(kd,0,\cdots,0),\ k\geq 0. \\
0, & \rm{if}\ \ll_2=1.
\end{cases}
\end{equation}

We write $\pi_{\ll}:S\to S^{\ll}$ for the natural projection map, and observe that multiplication on $S$ induces multiplication maps
\[m_{\ll,\mu}:S^{\ll}\oo S^{\mu}\lra S^{\ll+\mu},\rm{ via }S^{\ll}\oo S^{\mu}\subset S\oo S\lra S\overset{\pi_{\ll+\mu}}{\lra} S^{\ll+\mu}.\]
Let $\d=(d,0,\cdots,0)$ and note that $S^{\d}=\Sym^d V$ is the space of linear forms in $S$. For each $\ll$ we define $A^{\ll}$ as the image
\[A^{\ll}=\rm{Im}(S^{\ll-\d}\oo S^{\d}\overset{m_{\ll-\d,\d}}{\lra}S^{\ll}),\]
and let $a_{\ll}=\scpr{S_{\ll}V}{A^{\ll}}$. It follows from \cite{manivel} that
\begin{equation}\label{eq:multinj}
 a_{\ll}=s_{\ll-\d}.
\end{equation}
In fact, if $B^{\ll-\d}\subset S^{\ll-\d}$ is any subrepresentation, and if we let $B^{\ll}=m_{\ll-\d,\d}(B^{\ll-\d}\oo S^{\d})$ then
\begin{equation}\label{eq:multinjbll}
 \scpr{S_{\ll-\d}V}{B^{\ll-\d}}=\scpr{S_{\ll}V}{B^{\ll}},
\end{equation}
which is a consequence of the fact that the subring of unipotent invariants in $S$ is an integral domain.

It follows from (\ref{eq:multinj}), since $a_{\ll}\leq s_{\ll}$, that the function $f_{\ll}(k)=s_{\ll+k\d}$ is non-decreasing. Moreover,
\begin{equation}\label{eq:stab}
 a_{\ll+k\d}=s_{\ll+k\d}\rm{ for }k\gg 0
\end{equation}
which means that $f_{\ll}(k)$ stabilizes. The stable value is $f_{\ll}(k)=\nu_{\ol{\ll}}$, where $\ol{\ll}=(\ll_2,\cdots,\ll_n)$ (see (\ref{eq:defnumu})).

\subsection{Bott's theorem for projective space {\cite[Ch.~4]{weyman}}}\label{subsec:bott}

We consider $X=\bb{P}V$, the projective space of lines in $W$ (or 1-dimensional quotients of $V=W^*$), with the tautological sequence
\begin{equation}\label{eq:tautPV}
0\lra\mc{R}\lra V\oo\mc{O}_X\lra\mc{Q}\lra 0, 
\end{equation}
where $\mc{Q}=\mc{O}_X(1)$ is the tautological quotient bundle, and $\mc{R}=\Omega_X(1)$ is the tautological sub-bundle (here $\Omega_X=\Omega^1_X$ denotes the sheaf of differentials on $X$; later we will denote by $\Omega^i_X$ its $i$-th exterior power $\bw^i\Omega_X$). Bott's theorem gives a recipe to compute all the cohomology groups $H^{\bullet}(X,\mc{M})$ for a class of sheaves $\mc{M}$:

\begin{theorem}[Bott's Theorem]\label{thm:bott}
 Let $\mu\in\bb{Z}^{n-1}$ be a dominant weight and let $r\in\bb{Z}$ be an integer. If $r=\mu_i-i$ for some $i=1,\cdots,n-1$, then $H^j(X,S_{\mu}\mc{R}\oo\mc{Q}^r)=0$ for all $j=0,\cdots,n-1$. Otherwise, letting $\mu_0=\infty$ and $\mu_n=-\infty$, there exists a unique $l\in\{0,1,\cdots,n-1\}$ such that
\[\mu_l-l>r>\mu_{l+1}-(l+1).\]
Let $\ll\in\bb{Z}^n$ be the dominant weight defined by 
\[\ll=(\mu_1-1,\cdots,\mu_l-1,r+l,\mu_{l+1},\cdots,\mu_n).\]
We have
\[H^j(X,S_{\mu}\mc{R}\oo\mc{Q}^r)=\begin{cases}
 S_{\ll}V, & j=l; \\
 0, & j\neq l.
\end{cases}
\]
\end{theorem}

Given a dominant weight $\ll\in\bb{Z}^n$, define for $1\leq i\leq n$ 
\begin{equation}\label{eq:deftlll^i}
\tl{\ll}^i=(\ll_1+1,\cdots,\ll_{i-1}+1,\ll_{i+1},\cdots,\ll_n)\in\bb{Z}^{n-1}.
\end{equation}
Note that $\tl{\ll}^i=\ll^i+(u_d^{n-1})$ (with the notation in (\ref{eq:defll^i}) and (\ref{eq:defud})). Bott's Theorem then implies that
\begin{equation}\label{eq:Sll=Hj}
H^l(X,S_{\mu}\mc{R}\oo\mc{Q}^r)=S_{\ll}V\quad\Leftrightarrow\quad\mu=\tl{\ll}^{l+1}\rm{ and }r=\ll_{l+1}-l.
\end{equation}

\subsection{Admissible $\GL$-representations and cohomology}\label{subsec:SLFreps}

We consider as before a finite dimensional $\bb{C}$-vector space $W$ of dimension $\rm{dim}(W)=n$. We say that a $\GL(W)$--representation $M$ is \defi{admissible} if it isomorphic to a (possibly infinite) direct sum
\[M=\bigoplus_{\ll} (S_{\ll}W)^{\oplus a_{\ll}},\] where all the multiplicities are finite ($0\leq a_{\ll}<\infty$ for all $\ll$). $M$ is \defi{finite} if in addition only finitely many of the $a_{\ll}$'s are non-zero. 


The \defi{Grothendieck group} $\G(W)$ of admissible representations is the direct product of (infinitely many) copies of $\bb{Z}$, indexed by the collection $S_{\ll}W$, where $\ll$ runs over the set of dominant weights in $\bb{Z}^n$. We refer to the elements of $\G(W)$ as \defi{virtual representations}. We extend the notation introduced in (\ref{eq:defscpr}) to allow $U$ to be a virtual representation. As a corollary of the Littlewood-Richardson rule \cite[Section~I.9]{macdonald}, if $M$ is an admissible representation and $N$ is finite, then $M\oo N$ is also admissible.

Since any two admissible representations are isomorphic if and only if they coincide in $\G(W)$, we won't make any notational distinction between virtual and usual representations. When $0\to A\to B\to C\to 0$ is a short exact sequence of admissible representations, we have $B=A+C$ in $\G(W)$. Given a finite length complex $\mc{C}^{\bullet}$ of admissible representations, we define its \defi{Euler characteristic} to be the virtual representation given by
\[\chi(\mc{C}^{\bullet})=\sum_{i\in\bb{Z}}(-1)^i\mc{C}^i\in\G(W).\]

\begin{lemma}\label{lem:homiseulerchar}
 The Euler characteristic is preserved by taking homology, i.e. $\chi(\mc{C}^{\bullet})=\chi(H^{\bullet}(\mc{C}^{\bullet}))$. In particular, if $\mc{C}^{\bullet}$ has homology concentrated in a single degree~$i$, then we have the equality in $\G(W)$
\[H^i(\mc{C}^{\bullet})=(-1)^i\cdot\chi(\mc{C}^{\bullet}).\]
\end{lemma}

Suppose now that $X$ is a projective variety on which the group $\GL(W)$ acts. Assume further that $\mc{M}$ is a quasi-coherent $\GL(W)$-equivariant sheaf on $X$. We say that $\mc{M}$ has \defi{admissible cohomology} if its cohomology groups $H^j(X,\mc{M})$ are admissible for $j=0,\cdots,\dim(X)$. It will be useful to establish the following:

\begin{lemma}\label{lem:SLFgenerizes}
 Let $\mc{M}$ be a $\GL(W)$-equivariant quasi-coherent sheaf admitting a filtration $F_{\bullet}\mc{M}$ which is compatible with the $\GL$-action. If the associated graded $\gr(\mc{M})$ has admissible cohomology then so does~$\mc{M}$. Furthermore, we have an equality of Euler characteristics $\chi(R\pi_*\mc{M})=\chi(R\pi_*\gr(\mc{M}))$, where $\pi:X\to\rm{Spec}~\bb{C}$ denotes the structure morphism.
\end{lemma}

\begin{proof}
 There is a spectral sequence $E_2^{p,q}=H^{p-q}(X,F_q\mc{M}/F_{q+1}\mc{M})\Rightarrow H^{p-q}(X,\mc{M})$. Since $\gr(\mc{M})$ has admissible cohomology, for every dominant weight $\ll\in\bb{Z}^n$ the irreducible representation $S_{\ll}W$ appears with finite (total) multiplicity on the second page of the spectral sequence. The same conclusion must then be true on $E_{\infty}^{p,q}$, so $\mc{M}$ has admissible cohomology. The statement about Euler characteristics follows from Lemma~\ref{lem:homiseulerchar}, since each page $E_{i+1}^{\bullet,\bullet}$ is obtained by taking the homology of some complex whose terms are the terms on the previous page $E_i^{\bullet,\bullet}$.
\end{proof}

Assume now that $X=\bb{P}V$ is a projective space ($V=W^*$). We have

\begin{lemma}\label{lem:SLFcohomologyPV}
 Consider a quasi-coherent $\GL(W)$-equivariant sheaf $\mc{M}$ on $X=\bb{P}V$ which is a direct sum
\[\mc{M}=\bigoplus_{\substack{\mu=(\mu_1\geq\cdots\geq\mu_{n-1})\in\bb{Z}^{n-1} \\ r\in\bb{Z}}} (S_{\mu}\mc{R}\oo\mc{Q}^r)^{\oplus m_{\mu,r}},\]
where $m_{\mu,r}$ are non-negative integers. For any coherent $\mc{N}=\bigoplus_{\mu,r}(S_{\mu}\mc{R}\oo\mc{Q}^r)^{\oplus n_{\mu,r}}$ (i.e. finitely many $n_{\mu,r}$ are non-zero) we have that $\mc{M}\oo\mc{N}$ has admissible cohomology.
\end{lemma}

\begin{proof}
 This follows from the Littlewood-Richardson rule and Bott's Theorem~\ref{thm:bott}.
\end{proof}

We can define, in analogy to $\G(W)$, the group $\G(X,W)$ as a direct product of copies of $\bb{Z}$ indexed by $S_{\mu}\mc{R}\oo\mc{Q}^r$, with $\mu\in\bb{Z}^{n-1}$ dominant and $r\in\bb{Z}$. Every $\mc{M}$ as in Lemma~\ref{lem:SLFcohomologyPV} gives rise to an element in $\G(X,W)$, and such elements $\mc{M}$ generate the group. We can then define a natural map
\begin{equation}\label{eq:pfwd}
 \G(X,W)  \lra \G(W), \quad\quad \mc{M} \lra \chi(R\pi_*\mc{M}) = \sum_{j=0}^{n-1}(-1)^j H^j(X,\mc{M}).
\end{equation}

%
%

\subsection{$\D$-modules and local cohomology {\cite{borel}, \cite{hottaetal}, \cite{24hrs}}}\label{subsec:loccohDmods}

For a smooth algebraic variety $X$ over $\bb{C}$, we let $\D_X$ denote the sheaf of \defi{differential operators} on $X$ \cite[Section~1.1]{hottaetal}. A \defi{$\D$-module} $\mc{M}$ on $X$ (or a $\D_X$-module) is a quasi-coherent sheaf $\mc{M}$ with a left module action of $\D_X$. A basic example of a $\D$-module is the structure sheaf $\mc{O}_X$. When $X=\bb{C}^N$, $\mc{D}_X=\bb{C}[x_1,\cdots,x_N,\partial_1,\cdots,\partial_N]$ is the \defi{Weyl algebra} of differential operators with polynomial coefficients (where $\partial_i=\partial/\partial x_i$), and $\mc{O}_X=\bb{C}[x_1,\cdots,x_N]$.

All the $\D$-modules $\mc{M}$ that will concern us are going to be \defi{holonomic} \cite[Chapter~3]{hottaetal}, and in particular they will admit a finite composition series $\mc{M}=\mc{M}_0\supset\mc{M}_1\supset\cdots\supset\mc{M}_l=0$, where $\mc{Q}_i=\mc{M}_i/\mc{M}_{i+1}$ is a simple $\D$-module. The $\mc{Q}_i$'s are the \defi{composition factors} of $\mc{M}$ and are uniquely determined (up to reordering). When $X=\Spec\bb{C}$, a holonomic $\D$-module is just a finite dimensional vector space.

If $\pi:Y\lra X$ is a morphism between smooth varieties, there is a pushforward functor between the corresponding (derived) categories of $\D$-modules. Following \cite{hottaetal}, we denote this functor by $\int_{\pi}$ (in \cite{borel}, it is denoted $\pi_+$). By \cite[Theorem~3.2.3]{hottaetal}, $\int_{\pi}$ preserves holonomicity. When $\pi:X\to\Spec\bb{C}$ is the structure morphism, $\int_{\pi}\mc{M}$ is the hypercohomology of the \defi{de Rham complex}
\begin{equation}\label{eq:defdeRham}
dR(\mc{M}):\quad 0\lra\mc{M}\lra\Omega_X^1\oo\mc{M}\lra\cdots\lra\Omega_X^{\dim X}\oo\mc{M}\lra 0, 
\end{equation}
with differential given in local coordinates by $d(\om\oo m)=d\om\oo m+\sum_i (dx_i\wedge\om)\oo\partial_i m$. The hypercohomology groups of $dR(\mc{M})$ are called the \defi{de Rham cohomology groups} of $\mc{M}$, denoted $H^{\bullet}_{dR}(\mc{M})$, and are finite dimensional vector spaces when $\mc{M}$ is holonomic.

If $Y\subset X$ is a closed subvariety, the trivial local system on the smooth locus of $Y$ gives rise via the Riemann-Hilbert correspondence to a simple $\D_X$-module, which we denote by $\mc{L}(Y,X)$ (see \cite[Remark~7.2.10]{hottaetal}). Using the notation in the Introduction, for $Y=Z$ a Veronese cone and $X=\Sym^d W$, we have $\mc{L}(Y,X)=D_0$. If $Y$ is smooth and if we write $s:Y\hookrightarrow X$ for the inclusion map, then $\mc{L}(Y,X)=\int_s\mc{O}_Y$. If $X=\bb{C}^N$ and $Y=\{0\}$ then $\mc{L}(Y,X)=E$ is the $\D$-module discussed in the Introduction. We have
\[
 H^{\bullet}_{dR}(E)=\begin{cases}
\bb{C}, & \bullet=N; \\
0, & \rm{otherwise}.
\end{cases}
\]

\begin{lemma}\label{lem:topdR=0}
 If $M$ is a $\D$-module on $\bb{C}^N$ admitting a surjective map $p:M\to E$ then $H^N_{dR}(M)\neq 0$.
\end{lemma}

\begin{proof} This follows by writing down the long exact sequence of de Rham cohomology groups associated to the exact sequence $0\to\ker(p)\to M\to E\to 0$. Concretely, it follows from (\ref{eq:defdeRham}) that $H^N_{dR}(M)=M/\rm{par}(M)$, where $\rm{par}(M)$ is the subspace obtained by applying partial derivatives to the elements of $M$. The surjection $p$ induces a surjection $M/\rm{par}(M)\to E/\rm{par}(E)=\bb{C}$, from which the conclusion follows.
\end{proof}

The most important examples of holonomic $\D$-modules for us will be the \defi{local cohomology modules} $H^{\bullet}_Y(\mc{O}_X)$, where $Y\subset X$ is a closed subset. The local cohomology functor on $\D$-modules is discussed in \cite[Section~1.6]{hottaetal} and \cite[Section~VI.7]{borel}. The $\D$-module composition factors of local cohomology are typically hard to understand, but we have the following starting point:

\begin{proposition}\label{prop:loccohDmod}
 If $s:Y\hookrightarrow X$ is a closed immersion with $Y$ smooth, then $H^{\bullet}_Y(\mc{O}_X)=0$ for $\bullet\neq\rm{codim}(Y)$ and $H^{\rm{codim}(Y)}_Y(\mc{O}_X)=\mc{L}(Y,X)$.

 If $Y\subset X$ is not smooth, then $\mc{L}(Y,X)$ appears as a composition factor of $H^{\rm{codim}(Y)}_Y(\mc{O}_X)$ with multiplicity one. All the other simple factors appearing in the local cohomology modules $H^{\bullet}_Y(\mc{O}_X)$ are supported on the singular locus of $Y$.
\end{proposition}

\begin{proof} Assume first that $Y$ is smooth. By \cite[Proposition~1.7.1(iii)]{hottaetal} or \cite[Theorem~7.13(ii)]{borel}
\[H^{\bullet}_Y(\mc{O}_X)=\int_s s^{\dag}(\mc{O}_X),\]
where $s^{\dag}$ denotes the (shifted) inverse image functor. Since $s^{\dag}\mc{O}_X=\mc{O}_Y$ and $\int_s\mc{O}_Y=\mc{L}(Y,X)$, the conclusion follows.

If $Y$ is not smooth, then it follows from the previous discussion that away from the singular locus of $Y$, $H^{\bullet}_Y(\mc{O}_X)$ coincides with $\mc{L}(Y,X)$ in degree $\bullet=\rm{codim}(Y)$ and it vanishes otherwise. This means that the remaining composition factors of local cohomology are supported on the singular locus of $Y$.
\end{proof}

Specializing Proposition~\ref{prop:loccohDmod} to the case when $X=\Sym^d W$ and $Y=Z$ is the degree $d$ Veronese cone, it follows that $D_0$ appears as a composition factor of $H_Z^{\rm{codim}(Z)}(S)$ with multiplicity one, and the remaining composition factors (if any) of the local cohomology modules are supported at $0$, the vertex of the cone. By \cite[Example~1.6.4]{hottaetal}, the only simple $\D$-module supported at $0$ is the $\D$-module $E$ discussed before.

\subsection{Pushing forward $\D$-modules}\label{subsec:pfwd}

Consider a smooth projective variety $X$, a finite dimensional vector space $U$, and a short exact sequence
\begin{equation}\label{eq:basicses}
0\lra\xi\lra U\oo\mc{O}_X\lra\eta\lra 0, 
\end{equation}
where $\xi,\eta$ are locally free sheaves on $X$. We have a diagram
\begin{equation}\label{eq:diagrGeneric}
\xymatrix{
Y=\rm{Tot}(\eta^*) \ar@{^{(}->}[r]^{s} \ar[dr]_{\pi} & U^*\times X \ar[d]^{p} \\
 & U^* \\
}
\end{equation}
where $Y$ is the total space of the bundle $\eta^*$. All the $\D$-modules that we'll be concerned with will be assumed to be holonomic (even though some of the results below hold more generally). In particular, our modules will admit good filtrations in the sense of \cite[Section~2.1]{hottaetal} or \cite[Section~II.4]{borel}. We will be interested in understanding the (derived) $\D$-module pushforward $\int_{\pi}\mc{M}$ of a $\D_Y$-module $\mc{M}$ along the map $\pi$.

\begin{proposition}\label{prop:grIntM}
 Let $\mc{M}$ be a $\D_Y$-module with a good filtration. There exists a filtration on $\int_s\mc{M}$, such that the associated graded sheaves of $\mc{M}$ and $\int_s\mc{M}$ are related by
 \[\gr\left(\int_s\mc{M}\right)=\gr(\mc{M})\oo\det(\xi^*)\oo\Sym(\xi^*).\]
\end{proposition}

\begin{proof}
 Let $X'=U^*\times X$, and denote by $\mc{J}$ the ideal sheaf of $Y\subset X'$ which is generated by $\xi$ inside $\mc{O}_{X'}=\Sym_{\mc{O}_X}U$. The relative canonical sheaf $\om_{Y/X'}=\det(\mc{J}/\mc{J}^2)^*=\det(\xi^*)$ (or the pullback of $\det(\xi^*)$ from $X$ to $X'$ to be precise) and we have (see \cite[Section~VI.7]{borel})
\[\int_s\mc{M}=\D_{X'\leftarrow Y}\oo_{\D_Y}\mc{M},\]
where $\D_{X'\leftarrow Y}=(\D_{X'}/\mc{J}\D_{X'})\oo\om_{Y/X'}$. Consider the good filtration on $\mc{M}$ and the canonical filtrations on $\D_Y$ and $\D_{X'}$ (defined by the order of differential operators). We have $\gr(\D_Y)=\mc{O}_{T^*Y}$, where $T^*Y$ denotes (the total space of) the cotangent bundle on $Y$, and similarly $\gr(\D_{X'})=\mc{O}_{T^*X'}$. It follows that
\begin{equation}\label{eq:1stgrintM}
\begin{aligned}
\gr\left(\int_s\mc{M}\right)&=\gr(\D_{X'\leftarrow Y})\oo_{\gr(\D_Y)}\gr(\mc{M})=(\mc{O}_{T^*X'}/\mc{J}\mc{O}_{T^*X'})\oo_{\mc{O}_{T^*Y}}\gr(\mc{M})\oo\det(\xi^*)\\
&=\mc{O}_{T^*X'|_Y}\oo_{\mc{O}_{T^*Y}}\gr(\mc{M})\oo\det(\xi^*). 
\end{aligned}
\end{equation}
Since $(\xi)^*=(\mc{J}/\mc{J}^2)^*$ is the normal sheaf of the inclusion $Y\hookrightarrow X'$, we have an exact sequence
\begin{equation}\label{eq:ctgseq}
0\lra\mc{T}_Y\lra\mc{T}_{X'}|_Y\lra\xi^*\lra 0, 
\end{equation}
where $\mc{T}$ denotes the tangent sheaf. We have that $\mc{O}_{T^*X'|_Y}=\Sym_{\mc{O}_Y}(\mc{T}_{X'}|_Y)$ is locally a free module over $\mc{O}_{T^*Y}=\Sym_{\mc{O}_Y}(\mc{T}_Y)$. To make this global, we can then use the filtration  of $\mc{O}_{T^*X'|_Y}$ induced by (\ref{eq:ctgseq}) which yields the associated graded $\gr(\mc{O}_{T^*X'|_Y})=\Sym_{\mc{O}_{T^*Y}}(\xi^*)$. Combining this with (\ref{eq:1stgrintM}) we get (for a possibly different filtration of $\int_s\mc{M}$) that
\[\gr\left(\int_s\mc{M}\right)=\gr(\mc{M})\oo\det(\xi^*)\oo\Sym(\xi^*).\qedhere\]
\end{proof}

\begin{proposition}[{\cite[Cor.~5.3.2]{borel}}]\label{prop:pfwdprojection}
 If $\mc{N}$ is a $\D$-module on $U^*\times X$, then $\int_p\mc{N}=Rp_*(\om_X(\mc{N}))[d_X]$, where $d_X$ is the dimension of $X$, $[d_X]$ denotes the cohomological shift by $d_X$, and $\om_X(\mc{N})$ denotes the (relative) de Rham complex
\[\om_X(\mc{N}):\quad\quad 0\lra\mc{N}\lra\Omega^1_X\oo\mc{N}\lra\cdots\lra\Omega^{d_X}_X\oo\mc{N}\lra 0,\]
with $\mc{N}$ situated in cohomological degree $0$.
\end{proposition}

\begin{corollary}\label{cor:GrothClassPFwd} Suppose that $U$ is a $\GL(W)$-representation, that $X$ admits an action of $\GL(W)$ and that (\ref{eq:basicses}) is an exact sequence of $\GL(W)$-equivariant vector bundles. Assume further that for some good filtration on $\mc{M}$ and for $i=0,\cdots,d_X$, the sheaves $\Omega^i_X\oo\gr(\mc{M})\oo\det(\xi^*)\oo\Sym(\xi^*)$ have admissible cohomology. We have the following equality in $\G(W)$:
\[\chi\left(\int_{\pi}\mc{M}\right)=\sum_{i=0}^{d_X} (-1)^{d_X-i}\cdot\chi(R\pi_*(\Omega^i_X\oo\gr(\mc{M})\oo\det(\xi^*)\oo\Sym(\xi^*))).\]
\end{corollary}

\begin{proof}
 Let $\mc{N}=\int_s\mc{M}$ be the $\D$-module push-forward of $\mc{M}$ along the inclusion map $s$. It follows from Proposition~\ref{prop:grIntM} and Corollary~\ref{lem:SLFgenerizes} that the sheaves $\Omega_X^i\oo\mc{N}$ have admissible cohomology. We get from Proposition~\ref{prop:pfwdprojection} that the $\D$-module pushforward of $\mc{N}$ along $p$ is represented by a complex of admissible representations. Since $\int_p\mc{N}=\int_{\pi}\mc{M}$ the desired conclusion follows by taking Euler characteristics.
\end{proof}

\subsection{The modules $M_{\ll}$ and $\Ext$}\label{subsec:Mll}

Let $S=\Sym(\Sym^d V)$ with its natural $\GL$-action, and write $\d=(d,0,\cdots,0)$ as in Section~\ref{subsec:stablemults}. The $S$-modules $M_{\ll}$ introduced in the next lemma will play an essential role in the calculation of local cohomology in Section~\ref{sec:loccohVero}.

\begin{lemma}\label{lem:defMll}
Fix a partition $\ll$. There is a unique (up to isomorphism) $\GL$-equivariant $S$-module $M_{\ll}$ with the properties
\begin{enumerate}
 \item[(a)] As a $\GL$-representation, $M_{\ll}$ has a decomposition
\[M_{\ll}=\bigoplus_{k\geq 0}S_{\ll+k\d}V.\]
 \item[(b)] $M_{\ll}$ is generated as an $S$-module by the $\ll$-isotypic component $S_{\ll}V$.
\end{enumerate}
Assume now that for some $p>0$, we have a $\GL$-equivariant $S$-module $N$ which is isomorphic as a $\GL$-representation to $M_{\ll}^{\oplus p}$. If $N$ is generated as an $S$-module by its $\ll$-isotypic component (i.e. by the subrepresentation $(S_{\ll}V)^{\oplus p}$) then $N$ is isomorphic as an $S$-module to $M_{\ll}^{\oplus p}$.
\end{lemma}

\begin{proof}
 By (b), there is a surjective homomorphism $\pi:S_{\ll}V\oo S\lra M_{\ll}$, which is $\GL$-equivariant. Using the Littlewood-Richardson rule, we get for $k\geq 0$
\[\scpr{S_{\ll+k\d}V}{S_{\ll}V\oo S}=1\overset{(a)}{=}\scpr{S_{\ll+k\d}V}{M_{\ll}},\]
so the kernel $K$ of $\pi$ is the sum of the $\mu$-isotypic components of $S_{\ll}V\oo S$ corresponding to partitions $\mu\neq\ll+k\d$ for all $k$. It follows that $M_{\ll}\simeq(S_{\ll}V\oo S)/K$ is determined by properties (a) and (b).

The proof that $N$ is isomorphic to $M_{\ll}^{\oplus p}$ is identical to the argument given for the uniqueness of $M_{\ll}$.
\end{proof}

Our next goal is to compute the $\Ext$ modules $\Ext^{\bullet}_S(M_{\ll},S)$, and describe their $\GL$-equivariant structure. To do so, we will realize $M_{\ll}$ as the global sections of a locally free sheaf on projective space, and use the duality theorem \cite[Theorem~3.1]{raicu-weyman-witt} and Bott's Theorem~\ref{thm:bott} to compute the $\Ext$ modules.

Let $X=\bb{P}V$ be as in Section~\ref{subsec:bott}, with the tautological sequence (\ref{eq:tautPV}), and define $\xi$ to be the kernel of the natural map $\Sym^d V\oo\mc{O}_X\lra\mc{Q}^d$. For each partition $\ll$, let $\ol{\ll}=(\ll_2,\cdots,\ll_n)$ be the partition obtained by removing its largest part, and define
\[\mc{M}_{\ll}=S_{\ol{\ll}}\mc{R}\oo\mc{Q}^{\ll_1}\oo\Sym(\mc{Q}^d),\]
\[\mc{M}^*_{\ll}=S_{\ol{\ll}}\mc{R}\oo\mc{Q}^{\ll_1}\oo\det(\xi)\oo\Sym(\mc{Q}^{-d}).\]
By Bott's theorem and Lemma~\ref{lem:defMll}, we have $H^0(X,\mc{M}_{\ll})=M_{\ll}$, $H^j(X,\mc{M}_{\ll})=0$ for $j>0$. It follows from \cite[Theorem~3.1]{raicu-weyman-witt} that

\begin{proposition}\label{prop:ExtMll}
 If we let $n_d={n+d-1\choose d}-n$, then we have for every $j\in\bb{Z}$ and every partition $\mu$
\begin{equation}\label{eq:Ext=H}
\Ext^{n_d+j}_S(M_{\mu},S)=H^{n-1-j}(X,\mc{M}^*_{\mu})^*. 
\end{equation}
Assume that $\ll$ is a dominant weight, $\mu$ is a partition, and that $|\ll|\equiv|\mu|\equiv 0\ (\rm{mod }d)$. If $u_d$ is as in (\ref{eq:defud}), then setting $\ol{\mu}=(\mu_2,\cdots,\mu_n)$ and writing $\ll^i$ as in (\ref{eq:defll^i}), we get
\begin{itemize}
 \item[(i)] The irreducible representation $\det(\Sym^d W)=S_{(u_d^n)}W$ occurs in $\Ext^{\bullet}_S(M_{\mu},S)$ only if $\mu$ is a hook partition of size greater than zero (i.e. $\mu_2\leq 1\leq\mu_1$).
 \item[(ii)] For $j=0,1,\cdots,n-2$,
\[\scpr{S_{\ll}W}{\Ext^{n_d+j}(M_{\mu},S)}=\begin{cases}
 1, & \rm{if }\ol{\mu}=\ll^{n-j}.\\
 0, & \rm{otherwise}.
\end{cases}
\]
 \item[(iii)] For $j=n-1$,
\[\scpr{S_{\ll}W}{\Ext^{n_d+(n-1)}(M_{\mu},S)}=\begin{cases}
 1, & \rm{if }\ol{\mu}=\ll^{1}\rm{ and }\mu_1-\ll_1+u_d>0.\\
 0, & \rm{otherwise}.
\end{cases}
\]
\end{itemize}
\end{proposition}

\begin{proof}
 The equality (\ref{eq:Ext=H}) follows directly from \cite[Theorem~3.1]{raicu-weyman-witt}, if we note that $\rank(\xi)=n_d+(n-1)$. If $\ll=(u_d^n)$ then $\ll^{n-j}=(1,1,\cdots,1,0,0,\cdots,0)=(1^{n-1-j})$ for all $j=0,\cdots,n-1$. The condition $\ol{\mu}=\ll^{n-j}$ for some $j$ is equivalent to $\mu$ being a hook partition, so (i) follows from (ii) and (iii), which we verify next. We have
\[\scpr{S_{\ll}W}{\Ext^{n_d+j}(M_{\mu},S)}=\scpr{S_{\ll}V}{H^{n-1-j}(X,\mc{M}^*_{\mu})}=\scpr{S_{\ll-(u_d^n)}V}{H^{n-1-j}(X,S_{\ol{\mu}}\mc{R}\oo\mc{Q}^{\mu_1-d}\oo\Sym(\mc{Q}^{-d}))},\]
where the last equality follows from the fact that $\det(\xi)=\det(\Sym^d V)\oo\mc{Q}^{-d}=S_{(u_d^n)}V\oo\mc{Q}^{-d}$. We get using (\ref{eq:Sll=Hj}) that
\begin{equation}\label{eq:mult0or1}
\scpr{S_{\ll-(u_d^n)}V}{H^{n-1-j}(X,S_{\ol{\mu}}\mc{R}\oo\mc{Q}^{\mu_1-kd})}=0\rm{ or }1
\end{equation}
and it is equal to $1$ precisely when $\ol{\mu}=\ll^{n-j}$ and $\mu_1-kd=\ll_{n-j}-u_d-(n-1-j)$ (we apply (\ref{eq:Sll=Hj}) with $\ll$ replaced by $\ll-(u_d^n)$, $\mu$ replaced by $\ol{\mu}$, $l=n-1-j$, and $r=\mu_1-kd$). 

Since $\mc{Q}^{\mu_1-d}\oo\Sym(\mc{Q}^{-d})=\bigoplus_{k>0}\mc{Q}^{\mu_1-kd}$, in order to prove (ii) and (iii) we need to show that, under the assumption $\ol{\mu}=\ll^{n-j}$, there exists a positive integer $k$ such that $\mu_1-kd=\ll_{n-j}-u_d-(n-1-j)$, i.e. we have to show that
\[k=\frac{\mu_1-\ll_{n-j}+u_d+(n-1-j)}{d}\in\bb{Z}_{>0}.\]
The fact that $k\in\bb{Z}$ follows from the assumption that $|\ll|\equiv|\mu|\equiv 0\ (\rm{mod }d)$. When $j=n-1$, the positivity of $k$ is equivalent to $\mu_1-\ll_1+u_d>0$, so (iii) follows. When $0\leq j\leq n-2$ we have 
\[\mu_1\geq\mu_2\overset{(\ol{\mu}=\ll^{n-j})}{=}\ll_1+1-u_d > \ll_{n-j}-u_d,\]
which combined with $(n-1-j)>0$ yields the positivity of $k$.
\end{proof}

\section{Equivariant $\D$-modules on Veronese cones}\label{sec:equivDmodsVero}

Consider a finite dimensional complex vector space $W$ of dimension $n$, and write $V=W^*$ for its dual. The Veronese map $\Ver_d:W\to\Sym^d W$ of degree $d$ is defined by $\Ver_d(w)=w^d$. We write $Z$ for the image of this map, the \defi{degree $d$ Veronese cone}. The natural $\GL(W)$--action on $Z$ decomposes into two orbits $Z=\{0\}\cup(Z\setminus\{0\})$. Given any point $0\neq w^d\in Z$, the component group of its isotropy group is a a finite cyclic group of order $d$. Using the classification theorem for simple equivariant $\D$-modules \cite[Thm.~11.6.1]{hottaetal} together with \cite[Rem.~11.6.2]{hottaetal}, we find that there are $d$ simple equivariant $\D$-modules $D_0,\cdots,D_{d-1}$ with support $Z$, and one $\D$-module $E$ whose support is the origin. The latter is well understood (see the Introduction). The goal of this section is to describe $D_0,\cdots,D_{d-1}$ as $\GL(W)$-representations.

We consider the situation of Section~\ref{subsec:pfwd}, with $X=\bb{P}V$ and $\mc{R},\mc{Q}$ as in (\ref{eq:tautPV}). We let $U=\Sym^d V$, $\eta=\mc{Q}^d$, and define $\xi$ to be the kernel of the natural map $\Sym^d V\oo\mc{O}_X\lra\mc{Q}^d$. The diagram (\ref{eq:diagrGeneric}) becomes
\begin{equation}\label{eq:diagrPV}
\xymatrix{
Y=\rm{Tot}(\mc{Q}^{-d}) \ar@{^{(}->}[r]^{s} \ar[dr]_{\pi} & \Sym^d W\times\bb{P}V \ar[d]^{p} \\
 & \Sym^d W \\
}
\end{equation}
We let $S=\Sym(\Sym^d V)$ denote the ring of polynomial functions on $\Sym^d W$, and let $\mc{S}=\Sym(\mc{Q}^d)=\bigoplus_{i\geq 0}\mc{Q}^{id}$ be the sheaf of $\mc{O}_X$-algebras with the property that $\ul{\rm{Spec}}_{\mc{O}_X}\mc{S}=Y$ (we use the terminology of \cite[Exercise~II.5.17]{hartshorne}).

For $j\in\bb{Z}$, we consider the $\mc{S}$-modules
\begin{equation}\label{eq:defMscrj}
\mc{M}_j=\bigoplus_{i\in\bb{Z}}\mc{Q}^{di-j},
\end{equation}
and note that they are in fact $\mc{D}_Y$-modules which are $\GL(W)$--equivariant: if we write $x^{1/d}$ for a local generator of $\mc{Q}$, then $\mc{S}$ is locally isomorphic to the polynomial ring $\bb{C}[x]$, and $\mc{M}_j$ is locally isomorphic to the $\bb{C}[x]$-module $x^{-j/d}\cdot\bb{C}[x,1/x]$, which is also a module over the Weyl algebra $\bb{C}[x,\partial/\partial x]$; this local descriptions glue together to global $\D_Y$-modules $\mc{M}_j$. Note also that $\mc{M}_j=\mc{M}_{j'}$ if and only if $j\equiv j'\rm{ (mod }d)$. We start by showing that the $\D$-module pushforward $\int_{\pi}\mc{M}_j$ of each of the modules $\mc{M}_j$ along $\pi$ can be realized by a complex of admissible representations, and we compute the Euler characteristic of each of these complexes in the Grothendieck group $\G(W)$ (see Section~\ref{subsec:SLFreps}).

\begin{theorem}  For a dominant weight $\ll\in\bb{Z}^n$, define $m_{\ll}$ as in (\ref{eq:defmll}). We have
\[\chi\left(\int_{\pi}\mc{M}_j\right)=\sum_{|\ll|\equiv j\rm{ (mod }d)} m_{\ll}\cdot S_{\ll}W.\]
\end{theorem}

\begin{proof}
 For each $j$, the filtration induced by the direct sum decomposition of $\mc{M}_j$ is a good filtration. We get $\gr(\mc{M}_j)=\mc{M}_j$ and applying Proposition~\ref{prop:grIntM}, we have that $\int_s \mc{M}_j$ admits a filtration with associated graded
\[\rm{gr}\left(\int_s \mc{M}_j\right)=\mc{M}_j\oo\det(\xi^*)\oo\Sym(\xi^*).\]
We'd like to apply Corollary~\ref{cor:GrothClassPFwd} with $\mc{M}=\mc{M}_j$: in order to do so, we have to check that $\Omega^i_X\oo\mc{M}_j\oo\det(\xi^*)\oo\Sym(\xi^*)$ has admissible cohomology. By Lemma~\ref{lem:SLFgenerizes}, it suffices to prove this assertion after passing to an associated graded. $\xi^*$ has a filtration (see \cite[Exercise~II.5.16]{hartshorne}) with
\[\gr(\xi^*)=\bigoplus_{k=1}^d\Sym^k\mc{R}^*\oo\mc{Q}^{k-d},\]
which induces a filtration of $\Sym(\xi^*)$ with
\[\gr(\Sym(\xi^*))=\Sym\left(\bigoplus_{k=1}^d\Sym^k\mc{R}^*\oo\mc{Q}^{k-d}\right).\]
Applying Lemma~\ref{lem:SLFcohomologyPV} with $\mc{M}=\mc{M}_j\oo\gr(\Sym(\xi^*))$ and $\mc{N}=\Omega^i_X\oo\det(\xi^*)$ it follows that $\Omega^i_X\oo\mc{M}_j\oo\det(\xi^*)\oo\gr(\Sym(\xi^*))$ has admissible cohomology. Since $\dim(X)=n-1$, Corollary~\ref{cor:GrothClassPFwd} now yields
\[
\begin{split}
\chi\left(\int_{\pi}\mc{M}_j\right) = & \sum_{i=0}^{n-1} (-1)^{n-1-i}\cdot\chi(R\pi_*(\Omega^i_X\oo\mc{M}_j\oo\det(\xi^*)\oo\Sym(\xi^*))) \\
\overset{Lemma~\ref{lem:SLFgenerizes}}{=} & \sum_{i=0}^{n-1} (-1)^{n-1-i}\cdot\chi\left(R\pi_*\left(\Omega^i_X\oo\mc{M}_j\oo\det(\xi^*)\oo\Sym\left(\bigoplus_{k=1}^d\Sym^k\mc{R}^*\oo\mc{Q}^{k-d}\right)\right)\right )
\end{split}
\]
Using $\det(\xi^*)=\det(\Sym^d W)\oo\mc{Q}^d$ and $\mc{M}_j=\mc{M}_j\oo\mc{Q}^d$, we get $\mc{M}_j\oo\det(\xi^*)=\mc{M}_j\oo\det(\Sym^d W)$. Since $\Sym(\mc{A}\oplus\mc{B})=\Sym(\mc{A})\oo\Sym(\mc{B})$, we can rewrite the above equality as
\[\chi\left(\int_{\pi}\mc{M}_j\right) = \det(\Sym^d W)\oo\chi\left(R\pi_*\left(\Sym\left(\bigoplus_{k=2}^d\Sym^k\mc{R}^*\oo\mc{Q}^{k-d}\right)\oo\mc{N}\right)\right),\]
where $\mc{N}\in\G(X,W)$ is given by
\[\mc{N}=\sum_{i=0}^{n-1}(-1)^{n-1-i}\Omega^i_X\oo\mc{M}_j\oo\Sym(\mc{R}^*\oo\mc{Q}^{1-d})\]
Using again that $\mc{M}_j=\mc{M}_j\oo\mc{Q}^d$, we get that $\mc{M}_j\oo\Sym(\mc{R}^*\oo\mc{Q}^{1-d})=\mc{M}_j\oo\Sym(\mc{R}^*\oo\mc{Q})$. We have $\Omega^1_X=\mc{R}\oo\mc{Q}^*$, so $\Omega^i_X=\om_X\oo\bw^{n-1-i}(\mc{R}^*\oo\mc{Q})$, where $\om_X=\Omega^{n-1}_X$ is the canonical sheaf. It follows that
\[\mc{N}=\mc{M}_j\oo\om_X\oo\left(\sum_{i=0}^{n-1}(-1)^{n-1-i}\bw^{n-1-i}(\mc{R}^*\oo\mc{Q})\oo\Sym(\mc{R}^*\oo\mc{Q})\right)=\mc{M}_j\oo\om_X,\]
where the last equality follows from the fact that the Koszul complex on $\mc{R}^*\oo\mc{Q}$ resolves $\mc{O}_X$ as a module over the sheaf of $\mc{O}_X$-algebras $\Sym(\mc{R}^*\oo\mc{Q})$. It follows that
\[\chi\left(\int_{\pi}\mc{M}_j\right) = \det(\Sym^d W)\oo\chi\left(R\pi_*\left(\Sym\left(\bigoplus_{k=2}^d\Sym^k\mc{R}^*\oo\mc{Q}^{k-d}\right)\oo\mc{M}_j\oo\om_X\right)\right),\]
which by Serre duality is equivalent to
\begin{equation}\label{eq:chiMj1}
\chi\left(\int_{\pi}\mc{M}_j\right) = (-1)^{n-1}\cdot\det(\Sym^d W)\oo\chi\left(R\pi_*\left(\Sym\left(\bigoplus_{k=2}^d\Sym^k\mc{R}\oo\mc{Q}^{d-k}\right)\oo\mc{M}_{-j}\right)\right)^*. 
\end{equation}
(here we denoted by $*$ the duality operator on $\G(W)$ defined by $(S_{\ll}W)^*=S_{\ll}V=S_{(-\ll_n,\cdots,-\ll_1)}W$). For $k=2,\cdots,d$ we have the equality in $\G(X,W)$
\[\Sym(\Sym^k\mc{R}\oo\mc{Q}^{d-k})\oo\mc{M}_{-j}=\Sym(\Sym^k\mc{R}\oo\mc{Q}^{-k})\oo\mc{M}_{-j}=\Sym(\Sym^k(\mc{R}\oo\mc{Q}^*))\oo\mc{M}_{-j},\]
so
\begin{equation}\label{eq:chiMj2}
\begin{aligned}
\Sym\left(\bigoplus_{k=2}^d\Sym^k\mc{R}\oo\mc{Q}^{d-k}\right)&\oo\mc{M}_{-j} = \left(\bigotimes_{k=2}^d\Sym(\Sym^k(\mc{R}\oo\mc{Q}^*))\right)\oo\mc{M}_{-j} \\
\overset{(\ref{eq:defnumu})}{=} & \bigoplus_{\mu}(S_{\mu}\mc{R})^{\oplus\nu_{\mu}}\oo\mc{M}_{-j+|\mu|}.
\end{aligned}
\end{equation}

We are now ready to compute, for a dominant weight $\ll$, the multiplicity of $S_{\ll}W$ inside the virtual representation $\chi\left(\int_{\pi}\mc{M}_j\right)$. We have
\[
\begin{split}
 \scpr{S_{\ll}W}{\chi\left(\int_{\pi}\mc{M}_j\right)} \overset{(\ref{eq:defud}),(\ref{eq:chiMj1}),(\ref{eq:chiMj2})}{=} & (-1)^{n-1}\cdot\scpr{S_{\ll-(u_d^n)}V}{\chi\left( R\pi_*\left(\bigoplus_{\mu}(S_{\mu}\mc{R})^{\oplus\nu_{\mu}}\oo\mc{M}_{-j+|\mu|}\right)\right)} \\
 = & \sum_{l=0}^{n-1}(-1)^{n-1-l}\scpr{S_{\ll-(u_d^n)}V}{H^l\left(X,\bigoplus_{\mu}(S_{\mu}\mc{R})^{\oplus\nu_{\mu}}\oo\mc{M}_{-j+|\mu|}\right)} \\
\end{split}
\]
Using (\ref{eq:Sll=Hj}) we see that the only terms on the right hand side with a non-trivial contribution are the ones for which $\mu=\tl{\ll}^{l+1}-(u_d^{n-1})=\ll^{l+1}$, and $\mc{Q}^{\ll_{l+1}-u_d-l}$ appears inside $\mc{M}_{-j+|\mu|}$, i.e. $|\ll|\equiv j\rm{ (mod }d)$. For $\ll$ satisfying $|\ll|\equiv j\rm{ (mod }d)$, we thus get
\[\scpr{S_{\ll}W}{\chi\left(\int_{\pi}\mc{M}_j\right)} = \sum_{l=0}^{n-1}(-1)^{n-(l+1)}\nu_{\ll^{l+1}}\overset{(\ref{eq:defmll})}{=}m_{\ll}.\qedhere\]
\end{proof}
%

To finish the proof of the first part of Theorem~\ref{thm:main}, it remains to show the following equalities in the Grothendieck group $\G(W)$:
\[\chi\left(\int_{\pi}\mc{M}_j\right)=\begin{cases}
 D_0+(-1)^{n-1}\cdot E, & j=0; \\
 D_j, & j=1,\cdots,d-1.\\
\end{cases}
\]

We first deal with the case $j=0$: we have an exact sequence of $\D_Y$-modules
\[0\lra\mc{O}_Y\lra\mc{M}_0\lra\mc{E}\lra 0.\]
If we think of $X$ as a closed subset of $Y$, embedded by the zero section, and consider the open immersion $u:Y\setminus X\to Y$, then $\mc{M}_0=\int_u\mc{O}_{X\setminus Y}$ and $\mc{E}=H^1_X(\mc{O}_Y)$ is the first local cohomology sheaf of $\mc{O}_Y$ with support in $X$. This is a relative version of the exact sequence
\[0\lra\bb{C}[x]\lra\bb{C}[x,1/x]\lra H^1_{\{0\}}(\bb{C}[x])\lra 0\]
of $\D$-modules on the affine line $\bb{A}^1$. We get
\begin{equation}\label{eq:chiM0}
\chi\left(\int_{\pi}\mc{M}_0\right)=\chi\left(\int_{\pi}\mc{O}_Y\right)+\chi\left(\int_{\pi}\mc{E}\right). 
\end{equation}
$\mc{E}$ is supported on the exceptional divisor of $\pi$ (which we identified with $X$ via the $0$ section), and it corresponds via the Riemann-Hilbert correspondence to the intersection cohomology sheaf $IC_X$ on $X$, with respect to the trivial local system. Since $\pi$ contracts $X$ to $\{0\}$ and $X$ is smooth, $\int_{\pi}\mc{E}$ is described entirely in terms of the singular cohomology of $X$: it consists of $\dim(H^{2j}(X,\bb{C}))$ copies of $E$ in cohomological degree $2j-\dim(X)$, for each $j=0,\cdots,\dim(X)$. Since $X=\bb{P}V=\bb{P}^{n-1}$, we get $H^{2j}(X,\bb{C})=\bb{C}$, $H^{2j+1}(X,\bb{C})=0$, and therefore
\begin{equation}\label{eq:chiE}
\chi\left(\int_{\pi}\mc{E}\right)=(-1)^{n-1}\cdot n\cdot E.
\end{equation}
To compute $\chi\left(\int_{\pi}\mc{O}_Y\right)$ we need to understand explicitly the Decomposition Theorem for the map $\pi:Y\to X$. This is a special case of \cite[Theorem~6.1]{deCataldo-Migliorini-Mustata}, which in our case says that $\int_{\pi}\mc{O}_Y$ consists of one copy of $D_0$ in cohomological degree $0$, and one copy of $E$ in each of the cohomological degrees $n-2,n-4,\cdots,2-n$. It follows that
\begin{equation}\label{eq:chiOy}
\chi\left(\int_{\pi}\mc{O}_Y\right)=D_0+(-1)^{n-2}\cdot(n-1)\cdot E.
\end{equation}
The formula for $\chi\left(\int_{\pi}\mc{M}_0\right)$ now follows from (\ref{eq:chiM0}), (\ref{eq:chiE}) and (\ref{eq:chiOy}).

Now for $j=1,\cdots,d-1$ the Euler characteristics $\chi\left(\int_{\pi}\mc{M}_j\right)$ have no overlaps in terms of the $S_{\ll}W$'s that occur with non-zero multiplicity, and they must be described entirely in terms of the $\D$-modules $D_1,\cdots,D_{n-1}$, ($E$ and $D_0$ can't show up since their characters have weights of total size divisible by $d$). It follows that each $\chi\left(\int_{\pi}\mc{M}_j\right)=m_j\cdot D_j$ in $\G(W)$ for some $m_j\in\bb{Z}$. To show that $m_j=1$ it suffices to show that some multiplicity $m_{\ll}=\scpr{S_{\ll}W}{\chi\left(\int_{\pi}\mc{M}_j\right)}$ is equal to $1$. To do so, we choose $\ll\in\bb{Z}^n$ with
\[\ll_1=u_d+1,\quad\ll_2=\cdots=\ll_{n-1}=u_d-1,\ll_n<u_d,\]
and such that $|\ll|\equiv j\rm{ (mod }d)$. It follows that $\ll^n=(2,0,\cdots,0)$, $\nu_{\ll^n}=1$, and $\nu_{\ll^i}=0$ for $i<n$. This implies that $m_{\ll}=1$, as desired.

\section{Local cohomology with support in Veronese cones}\label{sec:loccohVero}

Let $Z$ denote as before the Veronese cone in $\Sym^d W$, and let $S=\Sym(\Sym^d V)$ be the ring of polynomial functions on the vector space $\Sym^d W$. Write $n_d={n-1+d\choose d}-n$ for the codimension of $Z$ inside $\Sym^d W$. In this section we prove the final part of Theorem~\ref{thm:main}:

\begin{theorem}\label{thm:loccoh}
 The local cohomology modules $H_Z^{\bullet}(S)$ of $S$ with support in $Z$ are given as follows:
\[
 H_Z^{\bullet}(S)=
 \begin{cases}
  D_0, & \bullet = n_d; \\
  0, & otherwise.
 \end{cases}
\]
\end{theorem}

\begin{proof} We summarize our proof strategy before proceeding to give more details:
\begin{itemize}
 \item[(a)] The theorem reduces by Proposition~\ref{prop:loccohDmod} and the paragraph following it to showing that $E$ doesn't occur as a composition factor of any of the local cohomology modules $H^{\bullet}_Z(S)$.
 \item[(b)] The local cohomology modules can be computed as a direct limit
\[H^{\bullet}_Z(S)=\varinjlim_{r}\Ext^{\bullet}_S(S/I_r,S)\]
where $(I_r)_{r\geq 0}$ is a system of ideals which is cofinal with the one consisting of the powers of the defining ideal $I_Z$ of $Z$ \cite[Ex.~A1D.1]{eis-syzygies}.
 \item[(c)] We choose a sequence of ideals $(I_r)_{r\geq 0}$ as in (b), such that each $I_r$ is $\GL$-equivariant, and each successive quotient $I_r/I_{r+1}$ is a direct sum of the modules $M_{\ll}$ studied in Section~\ref{subsec:Mll}. In particular, we know how to compute the $\Ext$ modules $\Ext^{\bullet}_S(I_r/I_{r+1},S)$.
 \item[(d)] It follows from (b) that there is a spectral sequence
 \[E^{p,q}_2=\Ext^{p-q}_S(I_q/I_{q+1},S)\Rightarrow H^{p-q}_Z(S).\]
 \item[(e)] $\det(\Sym^d W)$ appears as a subrepresentation of $E$, but it doesn't occur on the $E_2$ page, so it cannot occur in any of the local cohomology modules $H^{\bullet}_Z(S)$.
\end{itemize}

Parts (a) and (b) require no further explanations, so we start by constructing the ideals $I_r$. Recall the definition of $\d$, $A^{\ll}$, $S^{\ll}$, $a_{\ll}$, $s_{\ll}$ from Section~\ref{subsec:stablemults}. For each $\ll$, we choose a complement $P^{\ll}\subset S^{\ll}$ to $A^{\ll}$: $P^{\ll}$ is a subrepresentation of $S^{\ll}$ such that
\[S^{\ll}=A^{\ll}\oplus P^{\ll}.\]
We write $p_{\ll}=\scpr{S_{\ll}V}{P^{\ll}}=s_{\ll}-a_{\ll}$ and call the elements of $P^{\ll}$ \defi{primitive}. Note that the stabilization result of Manivel can be reformulated as follows: given $\ll$, 
\begin{equation}\label{eq:pll+kd=0}
p_{\ll+k\d}=0\rm{ for }k\gg 0,
\end{equation}
and moreover, if $\ol{\ll}=(\ll_2,\cdots,\ll_n)$ then
\begin{equation}\label{eq:sumpllkd}
\sum_{k\in\bb{Z}}p_{\ll+k\d}=\nu_{\ol{\ll}}. 
\end{equation}
For each $\ll$, we define $I_{\ll}$ to be the ideal generated by the elements in $P^{\ll}$:
\[I_{\ll}=(P^{\ll})\subset S.\]
We choose a total ordering of the partitions $\ll=(\ll_1,\cdots,\ll_n)$ for which $p_{\ll}\neq 0$,
\begin{equation}\label{eq:orderll}
\ll(0),\ll(1),\cdots,\ll(r),\cdots 
\end{equation}
satisfying the following properties:
\begin{itemize}
 \item[(1)] For $i<j$, $\ll(i)$ does not contain $\ll(j)$, i.e. there exists $k\in\{1,\cdots,n\}$ such that $\ll(i)_k<\ll(j)_k$.
 
 \item[(2)] The function $g(i)=\ll(i)_2+\cdots+\ll(i)_n$ that measures the sum of all but the first part of $\ll(i)$ is non-decreasing.
\end{itemize}
Note that there are infinitely many $\ll$'s with fixed $\ll_2+\cdots+\ll_n$, so one may wonder whether an ordering (\ref{eq:orderll}) exists which satisfies (2). However, the only partitions $\ll$ appearing in (\ref{eq:orderll}) are the ones for which $p_{\ll}\neq 0$, and it follows from (\ref{eq:pll+kd=0}) that after fixing $\ll_2+\cdots+\ll_n$, there are only finitely many such $\ll$'s.

We define a decreasing sequence of ideals
\[S=I_0\supset I_1\supset\cdots\supset I_r\supset\cdots,\rm{ by}\]
\[I_r=\sum_{i\geq r} I_{\ll(i)}.\]
Note that $I_r/I_{r+1}$ is generated by the (image of the) elements of $P^{\ll(r)}$. Moreover,
\begin{lemma}\label{lem:Ir/Ir+1}
 We have an isomorphism of $S$-modules
 \[I_r/I_{r+1}\simeq M_{\ll(r)}^{\oplus p_{\ll(r)}}.\]
\end{lemma}

\begin{proof} Since $I_r/I_{r+1}$ is generated by its $\ll(r)$-isotypic component, it follows from Lemma~\ref{lem:defMll} that it is enough to prove the isomorphism only as $\GL$-representations. It follows from (\ref{eq:multinjbll}) that
\[\scpr{S_{\ll(r)+k\d}V}{I_{\ll(r)}}=p_{\ll(r)}.\]
In order to show that $\scpr{S_{\ll(r)+k\d}V}{I_r/I_{r+1}}=p_{\ll(r)}$ we then have to show that for $i>r$
\begin{equation}\label{eq:llrlli0}
\scpr{S_{\ll(r)+k\d}V}{I_{\ll(r)}\cap I_{\ll(i)}}=0. 
\end{equation}
The only way $I_{\ll(i)}$ can have a non-trivial $(\ll(r)+k\d)$-isotypic component is if $\ll(i)$ was contained in $\ll(r)+k\d$. If $i>r$, then by conditions (1) and (2) this is is only possible if $\ll(i)=\ll(r)+k'\d$ for some $0<k'\leq k$. Since $I_{\ll(i)}$ is generated by primitive elements (i.e. elements that don't come from $I_{\ll(r)}$), (\ref{eq:llrlli0}) follows.

It remains to show that if $\mu\neq\ll(r)+k\d$ then $\scpr{S_{\mu}V}{I_r/I_{r+1}}=0$. If $S_{\mu}V$ appears in $I_r/I_{r+1}$ then $\mu$ contains $\ll(r)$, and since it is not of the form $\ll(r)+k\d$, then we must have $\mu_2+\cdots+\mu_n>\ll(r)_2+\cdots+\ll(r)_n$. By condition (2), $I_{r+1}$ contains the whole $\mu$-isotypic component of $S$, so $S_{\mu}V$ can't appear in $I_r/I_{r+1}$.
\end{proof}

Lemma~\ref{lem:Ir/Ir+1} proves (c), part (d) is clear, so it remains to check (e). By Proposition~\ref{prop:ExtMll}(i), $\det(\Sym^d W)$ only occurs as a subrepresentation in $\Ext^{\bullet}_S(M_{\mu},S)$ when $\mu$ is a hook partition of size greater than zero and divisible by $d$. However, if $\mu$ is a hook partition with $\mu_2=1$ then $p_{\mu}=0$ (in fact $s_{\mu}=0$ by (\ref{eq:sllhook})), so no $\ll(r)$ is equal to $\mu$. If $\mu=(kd,0,\cdots)=k\d$ for $k>0$ then $p_{\mu}=s_{\mu}-a_{\mu}=1-1=0$. Therefore  $\det(\Sym^d W)$ doesn't occur as a subrepresentation in $\Ext^{\bullet}_S(I_r/I_{r+1},S)$ for any $r$. This proves (e) and concludes the proof of Theorem~\ref{thm:loccoh}.
\end{proof}

Let's briefly see now how Theorem~\ref{thm:loccoh} allows for an alternative derivation of the formula for the character of $D_0$ in Theorem~\ref{thm:main}. It follows from Lemma~\ref{lem:Ir/Ir+1} that
\[\Ext^{\bullet}_S(I_r/I_{r+1},S)=\Ext^{\bullet}_S(M_{\ll(r)},S)^{\oplus p_{\ll(r)}},\]
so $\Ext^j_S(I_r/I_{r+1},S)$ vanishes outside the range $n_d\leq j\leq n_d+n-1$. By the spectral sequence in step (d) of the proof of Theorem~\ref{thm:loccoh}, we obtain the following equality in $\G(W)$:
\begin{equation}\label{eq:D0specseq}
D_0=\sum_{j=0}^{n-1}(-1)^j\cdot\sum_{\substack{\mu=(\mu_1\geq\cdots\geq\mu_n\geq 0)\\|\mu|\equiv 0\ (\rm{mod }d)}}p_{\mu}\cdot\Ext^{n_d+j}(M_{\mu},S). 
\end{equation}
For $j=0,\cdots,n-2$, it follows from Proposition~\ref{prop:ExtMll}(ii) that
\begin{equation}\label{eq:sumpmuext}
 \sum_{\mu}p_{\mu}\cdot\Ext^{n_d+j}(M_{\mu},S)=\sum_{\substack{\ll,\mu \\ \ol{\mu}=\ll^{n-j}}}p_{\mu}\cdot S_{\ll}W=\sum_{\ll}\left(\sum_{\ol{\mu}=\ll^{n-j}}p_{\mu}\right)\cdot S_{\ll}W\overset{(\ref{eq:sumpllkd})}{=}\sum_{\ll}\nu_{\ll^{n-j}}\cdot S_{\ll}W.
\end{equation}
Note that the conditions $\ol{\mu}=\ll^1$ and $\mu_1-\ll_1+u_d>0$ in Proposition~\ref{prop:ExtMll}(iii) are equivalent to $\mu=\ll-(u_d^n)+k\d$ for some $k>0$. It follows that for $j=n-1$ we get
\begin{equation}\label{eq:sumpmuextn-1}
\begin{aligned} 
 &\sum_{\mu}p_{\mu}\cdot\Ext^{n_d+n-1}(M_{\mu},S) = \sum_{\ll}\left(\sum_{\substack{k>0 \\ \mu=\ll-(u_d^n)+k\d}}p_{\mu}\right)\cdot S_{\ll}W \\
 =&\sum_{\ll}\left(\sum_{\substack{k\in\bb{Z} \\ \mu=\ll-(u_d^n)+k\d}}p_{\mu}-\sum_{\substack{k\leq 0 \\ \mu=\ll-(u_d^n)+k\d}}p_{\mu}\right)\cdot S_{\ll}W\overset{(\ref{eq:sumpllkd}),(\ref{eq:multinj}),(\ref{eq:multinjbll})}{=}\sum_{\ll}(\nu_{\ll^1}-e_{\ll})\cdot S_{\ll}W.
\end{aligned}
\end{equation}
It follows from (\ref{eq:D0specseq}), (\ref{eq:sumpmuext}) and (\ref{eq:sumpmuextn-1}) that
\[D_0=\sum_{\ll}\left(\sum_{j=0}^{n-2}(-1)^j\cdot\nu_{\ll^{n-j}}+(-1)^{n-1}\cdot(\nu_{\ll^1}-e_{\ll})\right)\cdot S_{\ll}W=\sum_{\ll}(m_{\ll}+(-1)^n e_{\ll})\cdot S_{\ll}W,\]
which is precisely the formula proved in Theorem~\ref{thm:main}.

We end with an alternative proof of Theorem~\ref{thm:loccoh}, suggested by Robin Hartshorne:

\begin{proof}[Alternative proof of Theorem~\ref{thm:loccoh}] Using \cite[Example~4.6]{ogus}, there exists only one non-vanishing local cohomology module, namely $H_Z^{n_d}(S)$. To prove the theorem it is then sufficient to check that $E$ does not appear as a composition factor of $H^{n_d}_Z(S)$. This will follow from Lemma~\ref{lem:topdR=0} as soon as we can show that the top de Rham cohomology group of $H^{n_d}_Z(S)$ is zero, which we do next.

We begin by relating the de Rham cohomology groups of $H^{n_d}_Z(S)$ to the de Rham homology groups of $Z$, denoted $H_{\bullet}^{dR}(Z)$ \cite[Section~II.3]{hartshorne-deRham}. We write $X=\Sym^d W$, $N=\dim(X)=n_d+n$, and recall that
\[H_i^{dR}(Z)=\bb{H}_Z^{2N-i}(X,dR(S)),\]
where $dR(S)$ is the de Rham complex (\ref{eq:defdeRham}), and $\bb{H}_Z^{\bullet}$ denotes hypercohomology with support in $Z$. The spectral sequence $E_2^{p,q}=H^p_{dR}(H_Z^q(S))\Rightarrow\bb{H}_Z^{p+q}(X,dR(S))$
degenerates since $H_Z^{n_d}(S)$ is the unique non-vanishing local cohomology module. We get $\bb{H}_Z^{p+n_d}(X,dR(S))=H^p_{dR}(H_Z^{n_d}(S))$ for all $p$, so
\begin{equation}\label{eq:deRhamcoh=hom}
H^N_{dR}(H_Z^{n_d}(S))=\bb{H}_Z^{N+n_d}(X,dR(S))=H_n^{dR}(Z). 
\end{equation}
Since $Z$ is the cone over the degree $d$ Veronese variety, which is abstractly isomorphic to $\bb{P}^{n-1}$, we get as in \cite[Proposition~III.3.2]{hartshorne-deRham} an exact sequence
\[\cdots\lra H_n^{dR}(\bb{P}^{n-1})\overset{\cap\xi}{\lra} H_{n-2}^{dR}(\bb{P}^{n-1})\lra H_n^{dR}(Z)\lra H_{n-1}^{dR}(\bb{P}^{n-1})\overset{\cap\xi}{\lra} H_{n-3}^{dR}(\bb{P}^{n-1})\lra\cdots.\]
Here $\xi\in H^2(\bb{P}^{n-1})$ denotes the hyperplane class, and the de Rham homology of $\bb{P}^{n-1}$ is the same as the singular homology, since $\bb{P}^{n-1}$ is smooth. Since $H_i(\bb{P}^{n-1})\overset{\cap\xi}{\lra} H_{i-2}(\bb{P}^{n-1})$ is an isomorphism (unless $i=0$ or $i=2n$), we get that $H_n^{dR}(Z)=0$, which combined with (\ref{eq:deRhamcoh=hom}) and Lemma~\ref{lem:topdR=0} concludes the proof of the theorem. 
\end{proof}

\section*{Acknowledgments} 
I am grateful to Sam Evens for discussions on admissible representations, to Robin Hartshorne for helpful exchanges related to de Rham and local (co)homology, to Andr\'as L\H orincz and Jerzy Weyman for pointing me to the literature around Levasseur's conjecture, and to Mircea Musta\c t\u a, Claudia Polini, Anurag Singh, Matteo Varbaro and Uli Walther for inspiring conversations. Experiments with the computer algebra software Macaulay2 \cite{M2} have provided numerous valuable insights. This work was supported by the National Science Foundation Grant No.~1458715.


\end{document}